\documentclass[reqno]{amsart}
\usepackage{amsmath,amsthm,amssymb,amscd}

\newtheorem{theorem}{Theorem}[section]
\newtheorem{lemma}[theorem]{Lemma}

\newtheorem{corollary}[theorem]{Corollary}
\theoremstyle{definition}

\theoremstyle{remark}

\numberwithin{equation}{section}
\allowdisplaybreaks
\begin{document}
\title[number of complete subgraphs of Peisert graphs]
{number of complete subgraphs of Peisert graphs and finite field hypergeometric functions}


\author{Anwita Bhowmik}
 \address{Department of Mathematics, Indian Institute of Technology Guwahati, North Guwahati, Guwahati-781039, Assam, INDIA}
\email{anwita@iitg.ac.in}
 \author{Rupam Barman}
 \address{Department of Mathematics, Indian Institute of Technology Guwahati, North Guwahati, Guwahati-781039, Assam, INDIA}
 \email{rupam@iitg.ac.in}


\date{9th May 2022}
\begin{abstract}
For a prime $p\equiv 3\pmod{4}$ and a positive integer $t$, let $q=p^{2t}$. Let $g$ be a primitive element of the finite field $\mathbb{F}_q$. 
The Peisert graph $P^\ast(q)$ is defined as the graph with vertex set $\mathbb{F}_q$ where $ab$ is an edge if and only if $a-b\in\langle g^4\rangle \cup g\langle g^4\rangle$. 
We provide a formula, in terms of finite field hypergeometric functions, for the number of complete subgraphs of order four contained in $P^\ast(q)$. 
We also give a new proof for the number of complete subgraphs of order three  contained in $P^\ast(q)$ by evaluating certain character sums. The computations for the number of complete subgraphs of order four are quite tedious, 
so we further give an asymptotic result for the number of complete subgraphs of any order $m$ in Peisert graphs.
\end{abstract}
\maketitle
\section{introduction and statements of results}
The arithmetic properties of Gauss and Jacobi sums have a very long history in number theory, with applications in Diophantine equations and the theory of $L$-functions. Recently, number theorists have obtained generalizations of classical hypergeometric functions that are assembled with these sums, and these functions have recently led to applications in graph theory. Here we make use of these functions, as developed by Greene, McCarthy, and Ono \cite{greene, greene2,mccarthy3, ono2} to study substructures in Peisert graphs, which are relatives of the well-studied Paley graphs.
\par The Paley graphs are a well-known family of undirected graphs constructed from the elements of a finite field. Named after Raymond Paley, they were introduced as graphs independently by Sachs in 1962 and Erd\H{o}s \& R\'enyi in 1963, 
inspired by the construction of Hadamard matrices in Paley's paper \cite{paleyp}. Let $q\equiv 1\pmod 4$ be a prime power. Then the Paley graph of order $q$ is the graph with vertex set as the finite field $\mathbb{F}_q$ and edges defined as, 
$ab$ is an edge if $a-b$ is a non-zero square in $\mathbb{F}_q$.
\par
It is natural to study the extent to which a graph exhibits symmetry. A graph is called \textit{symmetric} if, given any two edges $xy$ and $x_1y_1$, there exists a graph automorphism sending $x$ to $x_1$ and $y$ to $y_1$. 
Another kind of symmetry occurs if a graph is isomorphic to its complement, in which case the graph is called \textit{self-complementary}. While Sachs studied the self-complementarity properties of the Paley graphs, Erd\H{o}s \& R\'enyi were interested in their symmetries. 
It turns out that the Paley graphs are both self-complementary and symmetric.
\par 
It is a natural question to ask for the classification of all self-complementary and symmetric (SCS) graphs. In this direction, Chao's classification in \cite{chao} sheds light on the fact that the only such possible graphs of prime order are the Paley graphs. 
Zhang in \cite{zhang}, gave an algebraic characterization of SCS graphs using the classification of finite simple groups, although it did not follow whether one could find such graphs other than the Paley graphs. 
In 2001, Peisert gave a full description of SCS graphs as well as their automorphism groups in \cite{peisert}. He derived that there is another infinite family of SCS graphs apart from the Paley graphs, and, in addition, one more graph not belonging to any of the two former families. 
He constructed the $P^\ast$-graphs (which are now known as \textit{Peisert graphs}) as follows. For a prime $p\equiv 3\pmod{4}$ and a positive integer $t$, let $q=p^{2t}$. Let $g$ be a primitive element of the finite field $\mathbb{F}_q$, 
that is, $\mathbb{F}_q^\ast=\mathbb{F}_q\setminus\{0\}=\langle g\rangle$. Then the Peisert graph $P^\ast(q)$ is defined as the graph with vertex set $\mathbb{F}_q$ where $ab$ is an edge if and only if $a-b\in\langle g^4\rangle \cup g\langle g^4\rangle$. 
It is shown in \cite{peisert} that the definition is independent of the choice of $g$. It turns out that an edge is well defined, since $q\equiv 1\pmod 8$ implies that $-1\in\langle g^4\rangle$.
\par
We know that a complete subgraph, or a clique, in an undirected graph is a set of vertices such that every two distinct vertices in the set are adjacent. The number of vertices in the clique is called the order of the clique. 
Let $G^{(n)}$ denote a graph on $n$ vertices and let $\overline{G^{(n)}}$ be its complement. 
Let $k_m(G)$ denote the number of cliques of order $m$ in a graph $G$. Let $T_m(n)=\text{min}\left(k_m(G^{(n)})+ k_m(\overline{G^{(n)}})\right) $ where the minimum is taken over all graphs on $n$ vertices. 
Erd\H{o}s \cite{erdos}, Goodman \cite{goodman} and Thomason \cite{thomason} studied $T_m(n)$ for different values of $m$ and $n$. Here we note that the study of $T_m(n)$ can be linked to Ramsey theory. 
This is because, the diagonal Ramsey number $R(m,m)$ is the smallest positive integer $n$ such that $T_m(n)$ is positive. Also, for the function $k_m(G^{(n)})+ k_m(\overline{G^{(n)}})$ on graphs with $n=p$ vertices, $p$ being a prime,  
Paley graphs are minimal in certain ways; for example, in order to show that $R(4,4)$ is atleast $18$, the Paley graph with $17$ vertices acts as the only graph (upto isomorphism) such that $k_m(G^{(17)})+ k_m(\overline{G^{(17)}})=0$. 
What followed was a study on $k_m(G)$, $G$ being a Paley graph. Evans et al. \cite{evans1981number} and Atansov et al. \cite{atanasov2014certain} gave formulae for $k_4(G)$, where $G$ is a Paley graph with number of vertices a prime and a prime-power, 
respectively. One step ahead led to generalizations of Paley graphs by Lim and Praeger \cite{lim2006generalised}, and computing the number of cliques of orders $3$ and $4$ in those graphs by Dawsey and McCarthy \cite{dawsey}. 
Very recently, we \cite{BB} have defined \emph{Paley-type} graphs of order $n$ as follows. For a positive integer $n$, the Paley-type graph $G_n$ has the finite commutative ring $\mathbb{Z}_n$ as its vertex set and edges defined as, $ab$ is an edge if 
and only if $a-b\equiv x^2\pmod{n}$ for some unit $x$ of $\mathbb{Z}_n$. For primes $p\equiv 1\pmod{4}$ and any positive integer $\alpha$, we have also found the number of cliques of order $3$ and $4$ in the Paley-type graphs $G_{p^{\alpha}}$. 
\par 
The Peisert graphs lie in the class of SCS graphs alongwith Paley graphs, so it would serve as a good analogy to study the number of cliques in the former class too. There is no known formula for the number of cliques of order $4$ in 
Peisert graph $P^{\ast}(q)$. The main purpose of this paper is to provide a general formula for $k_4(P^\ast(q))$. In \cite{jamesalex2}, Alexander found the number of cliques of order $3$ using the properties that the 
Peisert graph are edge-transitive and that any pair of vertices connected by an edge have the same number of common neighbors (a graph being edge-transitive means that, given any two edges in the graph, there exists a graph automorphism sending one edge to the other). 
In this article, we follow a character-sum approach to compute the number of cliques of orders $3$ and $4$ in Peisert graphs. In the following theorem, we give a new proof for the number of cliques of orders $3$ in Peisert graphs by evaluating certain character sums.
\begin{theorem}\label{thm1}
Let $q=p^{2t}$, where $p\equiv 3\pmod 4$ is a prime and $t$ is a positive integer. Then, the number of cliques of order $3$ in the Peisert graph $P^{\ast}(q)$ is given by $$k_3(P^\ast(q))=\dfrac{q(q-1)(q-5)}{48}.$$
\end{theorem}
Next, we find the number of cliques of order $4$ in Peisert graphs. In this case, the character sums are difficult to evaluate. We use finite field hypergeometric functions to evaluate some of the character sums. Before we state our result on $k_4(P^\ast(q))$, 
we recall Greene's finite field hypergeometric functions from \cite{greene, greene2}. Let $p$ be an odd prime, and let $\mathbb{F}_q$ denote the finite field with $q$ elements, where $q=p^r, r\geq 1$.
Let $\widehat{\mathbb{F}_q^{\times}}$ be the group of all multiplicative
characters on $\mathbb{F}_q^{\times}$. We extend the domain of each $\chi\in \widehat{\mathbb{F}_q^{\times}}$ to $\mathbb{F}_q$ by 
setting $\chi(0)=0$ including the trivial character $\varepsilon$. For multiplicative characters $A$ and $B$ on $\mathbb{F}_q$, the binomial coefficient ${A \choose B}$ is defined by
\begin{align*}
{A \choose B}:=\frac{B(-1)}{q}J(A,\overline{B}),
\end{align*}
where $J(A, B)=\displaystyle \sum_{x \in \mathbb{F}_q}A(x)B(1-x)$ denotes the Jacobi sum and $\overline{B}$ is the character inverse of $B$. 
For a positive integer $n$, and $A_0,\ldots, A_n, B_1,\ldots, B_n\in \widehat{\mathbb{F}_q^{\times}}$, Greene \cite{greene, greene2} defined the ${_{n+1}}F_n$- finite field hypergeometric function over
$\mathbb{F}_q$ by
\begin{align*}
{_{n+1}}F_n\left(\begin{array}{cccc}
A_0, & A_1, & \ldots, & A_n\\
& B_1, & \ldots, & B_n
\end{array}\mid x \right)
:=\frac{q}{q-1}\sum_{\chi\in \widehat{\mathbb{F}_q^\times}}{A_0\chi \choose \chi}{A_1\chi \choose B_1\chi}
\cdots {A_n\chi \choose B_n\chi}\chi(x).
\end{align*}
For $n=2$, we recall the following result from \cite[Corollary 3.14]{greene}: 
$${_{3}}F_{2}\left(\begin{array}{ccc}
A, & B, & C \\
& D, & E
\end{array}| \lambda\right)=\sum\limits_{x,y\in\mathbb{F}_q}A\overline{E}(x)\overline{C}E(1-x)B(y)\overline{B}D(1-y)\overline{A}(x-\lambda y).$$
Some of the biggest motivations for studying finite field hypergeometric functions have been their connections with Fourier coefficients and eigenvalues of modular forms and with counting points on certain kinds of algebraic varieties. For example, Ono \cite{ono} gave formulae for the number of $\mathbb{F}_p$-points on elliptic curves in terms of special values of Greene's finite field hypergeometric functions. In \cite{ono2}, Ono wrote a beautiful chapter on finite field hypergeometric functions and mentioned several open problems on hypergeometric functions and their relations to modular forms and algebraic varieties. In recent times, many authors have studied and found solutions to some of the problems posed by Ono. 
\par Finite field hypergeometric functions are useful in the study of Paley graphs, see for example \cite{dawsey, wage}. In the following theorem, we express the number of cliques of order $4$ in Peisert graphs in terms of finite field hypergeometric functions.
\begin{theorem}\label{thm2}
	Let $p$ be a prime such that $p\equiv 3\pmod 4$. For a positive integer $t$, let $q=p^{2t}$. Let $q=u^2+2v^2$ for integers $u$ and $v$ such that $u\equiv 3\pmod 4$ and $p\nmid u$ when $p\equiv 3\pmod 8$.	If $\chi_4$ is a character of order $4$, then the number of cliques of order $4$ in the Peisert graph $P^{\ast}(q)$ is given by 
	\begin{align*}
	k_4(P^\ast(q))=\frac{q(q-1)}{3072}\left[2(q^2-20q+81)+2 u(-p)^t+3q^2\cdot {_{3}}F_{2}\left(\begin{array}{ccc}
	\hspace{-.12cm}\chi_4, &\hspace{-.14cm} \chi_4, &\hspace{-.14cm} \chi_4^3 \\
	& \hspace{-.14cm}\varepsilon, &\hspace{-.14cm} \varepsilon
	\end{array}| 1\right)
	 \right].
	\end{align*}
\end{theorem} 
Using Sage, we numerically verify Theorem $\ref{thm2}$ for certain values of $q$. We list some of the values in Table \ref{Table-1}.  We denote by ${_{3}}F_{2}(\cdot)$ the hypergeometric function appearing in Theorem \ref{thm2}.
\begin{table}[ht]
	\begin{center}
	\begin{tabular}{|c |c | c | c | c | c | c|}
			\hline
		$p$	&$q$ & $k_4(P^\ast(q))$ & $u$ & $q^2 \cdot {_{3}}F_{2}(\cdot)$ & $k_4(P^\ast(q))$ &${_{3}}F_{2}(\cdot)$\\
		&& (by Sage) &  & (by Sage) & (by Theorem \ref{thm2}) &\\\hline 
		$3$	&$9$ & $0$ &  $-1$ & $10$ & $0$& $0.1234\ldots$ \\
		$7$	&$49$ & $2156$ &  $7$ & $-30$ & $2156$& $-0.0123\ldots$\\
		$3$	&$81$ & $21060$ & $7$ & $-62$ & $21060$& $-0.0094\ldots$\\
		$11$	&$121$ & $116160$ & $7$ & $42$ & $116160$& $0.0028\ldots$\\
		$19$	&$361$ & $10515930$ &  $-17$ & $522$ & $10515930$& $0.0040\ldots$\\
		$23$	&$529$ & $49135636$ & $23$ & $930$ & $49135636$& $0.0033\ldots$\\
			\hline 			
		\end{tabular}
	\caption{Numerical data for Theorem \ref{thm2}}
\label{Table-1}
\end{center}
\end{table}
\par 
We note that the number of $3$-order cliques in the Peisert graph of order $q$ equals the number of $3$-order cliques in the Paley graph of the same order. 
The computations for the number of cliques of order $4$ are quite tedious, so we further give an asymptotic result in the following theorem, for the number of cliques of order $m$ in Peisert graphs, $m\geq 1$ being an integer.
\begin{theorem}\label{asym}
	Let $p$ be a prime such that $p\equiv 3\pmod 4$. For a positive integer $t$, let $q=p^{2t}$. For $m\geq 1$, let $k_m(P^\ast(q))$ denote the number of cliques of order $m$ in the Peisert graph $P^\ast(q)$. 
	Then $$\lim\limits_{q\to\infty}\dfrac{k_m(P^\ast(q))}{q^m}=\dfrac{1}{2^{{m}\choose_{2}}m!}.$$
\end{theorem}
\section{preliminaries and some lemmas}
We begin by fixing some notations. For a prime $p\equiv 3\pmod{4}$ and positive integer $t$, let $q=p^{2t}$. Let $g$ be a primitive element of the finite field $\mathbb{F}_q$, that is, $\mathbb{F}_q^\ast=\mathbb{F}_q\setminus\{0\}=\langle g\rangle$. 
Now, we fix a multiplicative character $\chi_4$ on $\mathbb{F}_q$ of order $4$ (which exists since $q\equiv 1\pmod 4$). Let $\varphi$ be the unique quadratic character on $\mathbb{F}_q$. Then, we have $\chi_4^2=\varphi$. 
Let $H=\langle g^4\rangle\cup g\langle g^4\rangle$. Since $H$ is the union of two cosets of $\langle g^4\rangle $ in $\langle g\rangle $, we see that $|H|=2\times \frac{q-1}{4}=\frac{q-1}{2}$. 
We recall that a vertex-transitive graph is a graph in which, given any two vertices in the graph, there exists some graph automorphism sending one of the vertices to the other. Peisert graphs being symmetric, are vertex-transitive. Also,
the subgraphs induced by $\langle g^4\rangle$ and $g\langle g^4\rangle$ are both vertex transitive: if $s,t$ are two elements of $\langle g^4\rangle$ (or $g\langle g^4\rangle$) then the map on the vertex set of $\langle g^4\rangle$ (or $g\langle g^4\rangle$) 
given by $x\longmapsto \frac{t}{s} x$ is an isomorphism sending $s$ to $t$. 
The subgraph of $P^\ast(q)$ induced by $H$ is denoted by $\langle H\rangle$.
\par 
Throughout the article, we fix $h=1-\chi_4(g)$. For $x\in\mathbb{F}_q^\ast$, we have the following:
\begin{align}\label{qq}
\frac{2+h\chi_4(x)+\overline{h}\overline{\chi_4}(x)}{4} = \left\{
\begin{array}{lll}
1, & \hbox{if $\chi_4(x)\in\{1,\chi_4(g)\}$;} \\
0, & \hbox{\text{otherwise.}}
\end{array}
\right.
\end{align}
We note here that for $x\neq 0$, $x\in H$ if and only if $\chi_4(x)=1$ or $\chi_4(x)=\chi_4(g)$.
\par 
We have the following lemma which will be used in proving the main results.
\begin{lemma}\label{rr}
	Let $q=p^{2t}$ where $p\equiv 3\pmod 4$ is a prime and $t$ is a positive integer. Let $\chi_4$ be a multiplicative character of order $4$ on $\mathbb{F}_q$, and let $\varphi$ be the unique quadratic character. Then, we have $J(\chi_4,\chi_4)=J(\chi_4,\varphi)=-(-p)^t$.
\end{lemma}
\begin{proof}
	By \cite[Proposition 1]{katre}, we have $J(\chi_4,\chi_4)=-(-p)^t$. We also note that by Theorem 2.1.4 and Theorem 3.2.1 of \cite{berndt}, where the results remain the same if we replace a prime by a prime power, 
	we see that $J(\chi_4,\varphi)=\chi_4(4)J(\chi_4,\chi_4)=a_4+ib_4$, where $a_4^2+b_4^2=q$ and $a_4\equiv -(\frac{q+1}{2})\pmod 4$. Hence, $a_4\equiv 1\pmod 4$ and $a_4=-(-p)^t,b_4=0$. Thus, we obtain $J(\chi_4,\varphi)=J(\chi_4,\chi_4)=-(-p)^t$.	
\end{proof}  
Next, we evaluate certain character sums in the following lemmas.
\begin{lemma}\label{lem1}
	Let $q\equiv 1\pmod 4$ be a prime power and let $\chi_4$ be a character on $\mathbb{F}_q$ of order $4$ such that $\chi_4(-1)=1$, and let $\varphi$ be the unique quadratic character. Let $a\in\mathbb{F}_q$ be such that $a\neq0,1$. 
	Then, $$\sum_{y\in\mathbb{F}_q}\chi_4((y-1)(y-a))=\varphi(a-1)J(\chi_4,\chi_4).$$
\end{lemma}
\begin{proof} We have 
	\begin{align*}
	&\sum_{y\in\mathbb{F}_q}\chi_4((y-1)(y-a))=\sum_{y'\in\mathbb{F}_q}\chi_4(y'(y'+1-a))\\
&=\sum_{y''\in\mathbb{F}_q}\chi_4((1-a)y'')\chi_4((1-a)(y''+1))
=\varphi(1-a)\sum_{y''\in\mathbb{F}_q}\chi_4(y''(y''+1))\\
	&=\varphi(1-a)\sum_{y''\in\mathbb{F}_q}\chi_4(-y''(-y''+1))\\
	&=\varphi(1-a)J(\chi_4,\chi_4),
	\end{align*}
	where we used the substitutions $y-1=y'$, $y''=y'(1-a)^{-1}$, and replaced $y''$ by $-y''$.
\end{proof}
\begin{lemma}\label{lem2}
	Let $q\equiv 1\pmod 4$ be a prime power and let $\chi_4$ be a character on $\mathbb{F}_q$ of order $4$ such that $\chi_4(-1)=1$. Let $a\in\mathbb{F}_q$ be such that $a\neq0,1$. Then, $$\sum_{y\in\mathbb{F}_q}\chi_4(y)\overline{\chi_4}(a-y)=-1.$$
\end{lemma}
\begin{proof} We have 
\begin{align*}
&\sum_{y\in\mathbb{F}_q}\chi_4(y)\overline{\chi_4}(a-y)=\sum_{y'\in\mathbb{F}_q}\chi_4(ay')\overline{\chi_4}(a-ay')\\
&=\sum_{y'\in\mathbb{F}_q}\chi_4(y')\overline{\chi_4}(1-y')\\
&=\sum_{y'\in\mathbb{F}_q}\chi_4\left(y'(1-y')^{-1}\right)\\
&=\sum_{y''\in\mathbb{F}_q, y''\neq -1}\chi_4(y'') =-1,
\end{align*}
where we used the substitutions $y' =y a^{-1}$ and $y'' =y'(1-y')^{-1}$, respectively.
\end{proof}
\begin{lemma}\label{lem3}
	Let $q=p^{2t}$, where $p\equiv 3\pmod 4$ is a prime and $t$ is a positive integer. Let $\chi_4$ be a character on $\mathbb{F}_q$ of order $4$ and let $\varphi$ be the unique quadratic character. 
	Let $J(\chi_4,\chi_4)=J(\chi_4,\varphi)=\rho$, where $\rho=-(-p)^t$. Then,
	\begin{align}\label{koro}
	\sum\limits_{x,y\in\mathbb{F}_q, x\neq 1}\overline{\chi_4}(x)\chi_4(y)\chi_4(1-y)\chi_4(x-y)=-2\rho	
	\end{align} and
	\begin{align}\label{koro1}
	\sum\limits_{x,y\in\mathbb{F}_q, x\neq 1}\overline{\chi_4}(x)\chi_4(y)\chi_4(1-y)\overline{\chi_4}(x-y)=1-\rho.
	\end{align}
\end{lemma}
\begin{proof}
	By Lemma \ref{lem2}, we have
	\begin{align*}
	\sum_{y\neq 0,1}\chi_4(y)\chi_4(1-y)\sum_{x\neq 0,1,y}\overline{\chi_4}(x)\chi_4(x-y)
	&=\sum_{y\neq 0,1}\chi_4(y)\chi_4(1-y)\left[-1-\chi_4(y-1) \right]\\
	&=-\rho-\sum_y \chi_4(y)\varphi(1-y)=-2\rho,
	\end{align*}
which proves \eqref{koro}. Next, using the substitution $x'=xy^{-1}$, we have
\begin{align}\label{sum-new}
\sum_x \overline{\chi_4}(x)\overline{\chi_4}(x-y)&=\sum_{x'} \overline{\chi_4}(x'y)\overline{\chi_4}(x'y-y)\notag \\
&=\varphi(y)\rho.	
\end{align}
So, using \eqref{sum-new}, we find that
	\begin{align*}
	&\sum_{y\neq 0,1}\chi_4(y)\chi_4(1-y)\sum_{x\neq 0,1,y}\overline{\chi_4}(x)\overline{\chi_4}(x-y)\\
	&=\sum_{y\neq 0,1}\chi_4(y)\chi_4(1-y)\left[\varphi(y)\rho-\overline{\chi_4}(y-1) \right]\\
	&=\rho\sum_y \overline{\chi_4}(y)\chi_4(1-y)-\sum_y \chi_4(y)\\
	&=-\rho+1.
	\end{align*}
	This completes the proof of the lemma.
\end{proof}
We need to evaluate several analogous character sums as in Lemma \ref{lem3}. To this end, we have the following two lemmas whose proofs merely involve Lemmas \ref{lem1} and \ref{lem2} (as in Lemma \ref{lem3}).
\begin{lemma}\label{lema1}
	Let $q=p^{2t}$, where $p\equiv 3\pmod 4$ is a prime and $t$ is a positive integer. Let $\chi_4$ be a character on $\mathbb{F}_q$ of order $4$ and let $\varphi$ be the unique quadratic character. Let $J(\chi_4,\chi_4)=J(\chi_4,\varphi)=\rho$, where $\rho=-(-p)^t$. 
	Then, we have 
	\begin{align*}
	&\sum\limits_{x,y\in\mathbb{F}_q, x\neq 1} \chi_4^{i_1}(y)\chi_4^{i_2}(1-y)\chi_4^{i_3}(x-y)\\
	&=\left\{
	\begin{array}{lll}
	-2\rho, & \hbox{if $(i_1, i_2, i_3)\in \{(1, 1, 1), (-1, -1, -1)\};$} \\
	2, & \hbox{if $(i_1, i_2, i_3)\in \{(1, 1, -1), (-1, -1, 1)\};$} \\
	1-\rho, & \hbox{if $(i_1, i_2, i_3)\in \{(1, -1, 1), (1, -1, -1), (-1, 1, 1), (-1, 1, -1)\}$.}
	\end{array}
	\right.
	\end{align*}	
\end{lemma}
\begin{lemma}\label{corr}
	Let $q=p^{2t}$, where $p\equiv 3\pmod 4$ is a prime and $t$ is a positive integer. Let $\chi_4$ be a character on $\mathbb{F}_q$ of order $4$ and let $\varphi$ be the unique quadratic character. Let $J(\chi_4,\chi_4)=J(\chi_4,\varphi)=\rho$, where $\rho=-(-p)^t$. 
	Then, for $i_1,i_2,i_3\in\{\pm 1\}$, we have the following tabulation of the values of the expression given below:
	\begin{align}\label{new-eqn1}
	\sum\limits_{x,y\in\mathbb{F}_q, x\neq 1}A_x \cdot \chi_4^{i_1}(y)\chi_4^{i_2}(1-y)\chi_4^{i_3}(x-y).
	\end{align}
	For $w\in\{1,2,\ldots,8\}$ and $z\in \{1,2,\ldots,7\}$, the $(w,z)$-th entry in the table corresponds to \eqref{new-eqn1},
	where $A_x$ is either $\chi_4(x),\overline{\chi_4}(x),\chi_4(1-x)$ or $\overline{\chi_4}(1-x)$ and the tuple $(i_1,i_2,i_3)$ depends on $w$.
	\begin{align*}
	\begin{array}{|l|l|l|l|l|l|l|}
	\cline {4 - 7 } \multicolumn{3}{c|}{} & \multicolumn{4}{|c|}{A_{x}} \\
	\hline i_{1} & i_{2} & i_{3} & \chi_4(x) & \overline{\chi_4}(x) & \chi_4(1-x) & \overline{\chi_4}(1-x) \\
	\hline
	1 & 1 & 1  & -2 \rho & -2 \rho & -2 \rho & -2 \rho \\
	1 & 1 & -1  & 1-\rho & 1-\rho & 1- \rho & 1-\rho \\
	1 & -1 & 1 & {\rho}^2+1 & 2 & {\rho}^2-\rho & 1-\rho \\
	1 & -1 & -1 & 1-\rho & {\rho}^2-\rho & 2 & {\rho}^2+1 \\
	-1 & 1 & 1 &{\rho}^2-\rho  & 1-\rho &{\rho}^2+1  & 2 \\
	-1 & 1 & -1 &2  & {\rho}^2+1 &1-\rho  & {\rho}^2-\rho \\
	-1 & -1 & 1 &1-\rho  & 1-\rho &1-\rho  & 1-\rho \\
	-1 & -1 & -1 &-2\rho  & -2\rho &-2\rho  & -2\rho\\
	\hline 
	\end{array}	
	\end{align*}
	For example, the $(3,6)$-th position contains the value ${\rho}^2-\rho$. Here $w=3$ corresponds to $i_1=1,i_2=-1,i_3=1$; $z=6$ corresponds to the column $A_x=\chi_4(1-x)$. 
	So,  $$\sum\limits_{x,y\in\mathbb{F}_q, x\neq 1}\chi_4(1-x)\chi_4(y) \overline{\chi_4}(1-y)\chi_4(x-y)={\rho}^2-\rho.$$
\end{lemma}
\begin{proof}
	The calculations follow along the lines of Lemma \ref{lem1} and Lemma \ref{lem2}. For example, in Lemma \ref{lem3}, one can take $\chi_4(x),~\chi_4(x-1)$ or $\overline{\chi_4}(x-1)$ in place of $\overline{\chi_4}(x)$ in $\eqref{koro}$ and $\eqref{koro1}$ (which we denote by $A_x$), 
	and easily evaluate the corresponding character sum. 
\end{proof}
\begin{lemma}\label{lem4}
	Let $q=p^{2t}$, where $p\equiv 3\pmod 4$ is a prime and $t$ is a positive integer. Let $\chi_4$ be a character of order $4$. Let $\varphi$ and $\varepsilon$ be the quadratic and the trivial characters, respectively. Let $q=u^2+2v^2$ for integers $u$ and $v$ such that $u\equiv 3\pmod 4$ and $p\nmid u$ when $p\equiv 3\pmod 8$.
	Then,	
	\begin{align*}
	&{_{3}}F_2\left(\begin{array}{ccc}
	\chi_4, & \chi_4, & \chi_4\\ & \varepsilon, & \varepsilon
	\end{array}\mid 1 \right)=
	{_{3}}F_2\left(\begin{array}{ccc}\overline{\chi_4}, & \overline{\chi_4}, & \overline{\chi_4}\\ & \varepsilon, & \varepsilon \end{array}\mid 1\right)\\
	&={_{3}}F_2\left(\begin{array}{ccc}\chi_4, & \overline{\chi_4}, & \overline{\chi_4}\\ & \varphi, & \varepsilon\end{array}\mid 1\right)
	={_{3}}F_2\left(\begin{array}{ccc}\overline{\chi_4}, & \chi_4, & \chi_4\\ & \varphi, & \varepsilon\end{array}\mid 1\right)\\
	&=\frac{1}{q^2}[-2u(-p)^t].
	\end{align*}
\end{lemma}
\begin{proof}
	Let $\chi_8$ be a character of order $8$ such that $\chi_8^2=\chi_4$. Now, Proposition 1 in \cite{katre} tells us that $J(\chi_4,\chi_4)=-(-p)^t$ and hence it is real. Again, by Theorem 3.3.3 and the paragraph preceeding Theorem 3.3.1 in \cite{berndt}, $J(\chi_8,\chi_8^2)=\chi_8(-4)J(\chi_4,\chi_4)$, 
	where $\chi_8(4)=\pm 1$ and thus, is also real. By \cite[Theorem 4.37]{greene}, we have
	\begin{align}\label{doe}
	{_{3}}F_{2}\left(\begin{array}{ccc}\chi_4, & \chi_4, & \chi_4\\ & \varepsilon, & \varepsilon\end{array}\mid 1\right)&=\binom{\chi_8}{\chi_8^2}\binom{\chi_8}{\chi_8^3}+\binom{\chi_8^5}{\chi_8^2}\binom{\chi_8^5}{\overline{\chi_8}}\notag \\
	&=\frac{\chi_8(-1)}{q^2}[J(\chi_8,\chi_8^6)J(\chi_8,\chi_8^5)+J(\chi_8^5,\chi_8^6)J(\chi_8^5,\chi_8)].
	\end{align}
	Using Theorems 2.1.5 and 2.1.6 in \cite{berndt} we obtain
	\begin{align*}
	&J(\chi_8,\chi_8^6)=\chi_8(-1)J(\chi_8,\chi_8),\\
	&J(\chi_8,\chi_8^5)=\chi_8(-1)J(\chi_8,\chi_8^2),\\
	&J(\chi_8^5,\chi_8^6)=\chi_8(-1)\overline{J(\chi_8,\chi_8)}.	
	\end{align*}
	Substituting these values in $\eqref{doe}$ and using \cite[Lemma 3.6 (2)]{dawsey}, we find that 
	\begin{align}\label{real}
		{_{3}}F_{2}\left(\begin{array}{ccc}\chi_4, & \chi_4, & \chi_4\\ & \varepsilon, & \varepsilon\end{array}\mid 1\right)&=\frac{\chi_8(-1)}{q^2}[J(\chi_8,\chi_8)J(\chi_8,\chi_8^2)+\overline{J(\chi_8,\chi_8)}J(\chi_8,\chi_8^2)]\notag \\
		&=\frac{1}{q^2}J(\chi_8,\chi_8^2)\times 2 Re(J(\chi_8,\chi_8))\times \chi_8(-1)\notag \\
		&=\frac{1}{q^2}[-2u(-p)^t].
	\end{align}
Since ${_{3}}F_{2}\left(\begin{array}{ccc}\overline{\chi_4}, & \overline{\chi_4}, & \overline{\chi_4}\\ & \varepsilon, & \varepsilon\end{array}\mid 1\right)$ is the conjugate of 
${_{3}}F_{2}\left(\begin{array}{ccc}\chi_4, & \chi_4, & \chi_4\\ & \varepsilon, & \varepsilon\end{array}\mid 1\right)$, so both are equal as the value given in \eqref{real} is a real number. Using Lemma 4.37 in \cite{greene} again, we have
\begin{align}\label{jack}
{_{3}}F_{2}\left(\begin{array}{ccc}\chi_4, & \overline{\chi_4}, & \overline{\chi_4}\\ & \varphi, & \varepsilon\end{array}| 1\right)&=\binom{\overline{\chi_8}}{\chi_8^2}\binom{\overline{\chi_8}}{\chi_8}+\binom{\chi_8^3}{\chi_8^2}\binom{\chi_8^3}{\overline{\chi_8}^3}\notag \\
&=\frac{\chi_8(-1)}{q^2}[J(\overline{\chi_8},\overline{\chi_8}^2)J(\overline{\chi_8},\overline{\chi_8})+J(\chi_8^3,\overline{\chi_8}^2)J(\chi_8^3,\chi_8^3)].	
\end{align}
Recalling Theorem 2.1.6 in \cite{berndt} gives $J(\chi_8,\chi_8)=J(\chi_8^3,\chi_8^3)$. Also, Theorem 2.1.5 in \cite{berndt} gives
 $J(\chi_8^3,\overline{\chi_8}^2)=\overline{J(\chi_8^5,\chi_8^2)}=\overline{J(\chi_8,\chi_8^2)}=J(\chi_8,\chi_8^2)$. Hence, $\eqref{jack}$ yields
\begin{align*}
	{_{3}}F_{2}\left(\begin{array}{ccc}\chi_4, & \overline{\chi_4}, & \overline{\chi_4}\\ & \varphi, & \varepsilon\end{array}| 1\right)&=\frac{1}{q^2}J(\chi_8,\chi_8^2)\times 2 Re(J(\chi_8,\chi_8))\times \chi_8(-1)\\
	&=\frac{1}{q^2}[-2u(-p)^t],
\end{align*}	
 which is the same real number we found in $\eqref{real}$. Hence, its complex conjugate, namely ${_{3}}F_{2}\left(\begin{array}{ccc}\overline{\chi_4}, & \chi_4, & \chi_4\\ & \varphi, & \varepsilon\end{array}| 1\right)$ is also real and has the same value. This completes the proof of the lemma.
\end{proof}
 Next, we note the following observations given in the beginning of the sixth section in \cite{dawsey}. We state it as a lemma since we shall use it in proving Theorem \ref{thm2}. Greene \cite{greene, greene2} gave some transformation formulae which we 
 list here as follows. Let $A,B,C,D,E$ be characters on $\mathbb{F}_q$. Then, we have
 \begin{align}
 	&{_{3}}F_{2}\left(\begin{array}{ccc}A, & B, & C\\ & D, & E \end{array}| 1\right)={_{3}}F_{2}\left(\begin{array}{ccc}B\overline{D}, & A\overline{D}, & C\overline{D}\\ & \overline{D}, & E\overline{D}\end{array}| 1\right),\label{1}\\
 	&{_{3}}F_{2}\left(\begin{array}{ccc}A, & B, & C\\ & D, & E \end{array}| 1\right)=ABCDE(-1)\cdot {_{3}}F_{2}\left(\begin{array}{ccc}A, & A\overline{D}, & A\overline{E}\\ & A\overline{B}, & A\overline{C}\end{array}| 1\right),\label{2}\\
	&{_{3}}F_{2}\left(\begin{array}{ccc}A, & B, & C\\ & D, & E \end{array}|1\right)=ABCDE(-1)\cdot{_{3}}F_{2}\left(\begin{array}{ccc}B\overline{D}, & B, & B\overline{E}\\ & B\overline{A}, & B\overline{C}\end{array}| 1\right),\label{3}\\
 	&{_{3}}F_{2}\left(\begin{array}{ccc}A, & B, & C\\ & D, & E\end{array}| 1\right)=AE(-1)\cdot{_{3}}F_{2}\left(\begin{array}{ccc}A, & B, & E\overline{C}\\ & AB\overline{D}, & E\end{array}| 1\right),\label{4}\\
 	&{_{3}}F_{2}\left(\begin{array}{ccc}A, & B, & C\\ & D, & E\end{array}| 1\right)=AD(-1)\cdot{_{3}}F_{2}\left(\begin{array}{ccc}A, & D\overline{B}, & C\\ & D, & AC\overline{E} \end{array}| 1\right),\label{5}\\
 	&{_{3}}F_{2}\left(\begin{array}{ccc}A, & B, & C\\ & D, & E \end{array}| 1\right)=B(-1)\cdot{_{3}}F_{2}\left(\begin{array}{ccc}\overline{A}D, & B, & C\\ & D, & BC\overline{E}\end{array}| 1\right),\label{6}\\
 	&{_{3}}F_{2}\left(\begin{array}{ccc}A, & B, & C\\ & D, & E \end{array}| 1\right)=AB(-1)\cdot{_{3}}F_{2}\left(\begin{array}{ccc}\overline{A}D, & \overline{B}D, & C\\ & D, & DE\overline{AB}\end{array}| 1\right).\label{7}
 \end{align}
 Let $X=\{(t_1,t_2,t_3,t_4,t_5)\in\mathbb{Z}_4^5: t_1,t_2,t_3\neq 0,t_4,t_5;~t_1+t_2+t_3\neq t_4,t_5\}$. To each of the transformations in $\eqref{1}$ to $\eqref{7}$, Dawsey and McCarthy in \cite{dawsey} associated a map  on $X$; for example, the transformation in $\eqref{1}$ gives that
 $${_{3}}F_{2}\left(\begin{array}{ccc}\chi_4^{t_1}, & \chi_4^{t_2}, & \chi_4^{t_3}\\ & \chi_4^{t_4}, & \chi_4^{t_5}\end{array}| 1\right)={_{3}}F_{2}\left(\begin{array}{ccc}\chi_4^{t_2-t_4}, &\chi_4^{t_1-t_4} , & \chi_4^{t_3-t_4}\\ & \chi_4^{-t_4}, &\chi_4^{t_5-t_4}\end{array}| 1\right),$$
 so it induces a map $f_1: X\rightarrow X$
 given by
 $$f_1(t_1,t_2,t_3,t_4,t_5)=(t_2-t_4,t_1-t_4,t_3-t_4,-t_4,t_5-t_4).$$
 Similarly, the other transformations in $\eqref{2}$ to $\eqref{7}$ led to the construction of the maps $f_2$ to $f_7$. 
\begin{lemma}\label{dlemma}
	Let $X=\{(t_1,t_2,t_3,t_4,t_5)\in\mathbb{Z}_4^5: t_1,t_2,t_3\neq 0,t_4,t_5;~t_1+t_2+t_3\neq t_4,t_5\}$. Define the functions $f_i:X\rightarrow X,~i\in\{1,2,\ldots,7\}$ in the following manner:
	\begin{align*}
	f_1(t_1,t_2,t_3,t_4,t_5)&=(t_2-t_4,t_1-t_4,t_3-t_4,-t_4,t_5-t_4),\\	
f_{2}\left(t_{1}, t_{2}, t_{3}, t_{4}, t_{5}\right)&=\left(t_{1}, t_{1}-t_{4}, t_{1}-t_{5}, t_{1}-t_{2}, t_{1}-t_{3}\right),\\
f_{3}\left(t_{1}, t_{2}, t_{3}, t_{4}, t_{5}\right)&=\left(t_{2}-t_{4}, t_{2}, t_{2}-t_{5}, t_{2}-t_1,t_2-t_3\right),\\
f_{4}\left(t_{1}, t_{2}, t_{3}, t_{4}, t_{5}\right)&=\left(t_{1}, t_{2}, t_{5}-t_{3}, t_{1}+t_{2}-t_{4}, t_{5}\right),\\
f_{5}\left(t_1, t_{2}, t_{3}, t_{4}, t_{5}\right)&=\left(t_{1}, t_{4}-t_{2}, t_{3}, t_{4}, t_{1}+t_{3}-t_{5}\right),\\
f_{6}\left(t_{1}, t_{2}, t_{3}, t_{4}, t_{5}\right)&=\left(t_{4}-t_{1}, t_{2}, t_{3}, t_{4}, t_{2}+t_{3}-t_{5}\right),\\
f_{7}\left(t_{1},t_{2}, t_{3},t_{4}, t_{5}\right)&=\left(t_{4}-t_{1},t_{4}-t_{2},t_{3}, t_{4}, t_{4}+t_{5}-t_{1}-t_{2}\right).
\end{align*}
Then the group generated by $f_1,\ldots,f_7$, with operation composition of functions, is the set
$$\mathcal{F}=\{f_0,f_i,f_j \circ f_l,f_4\circ f_1,f_6\circ f_2,f_5\circ f_3,f_1\circ f_4\circ f_1: 1\leq i\leq 7,~1\leq j\leq 3,~4\leq l\leq 7\},$$
where $f_0$ is the identity map. \\
Moreover, the group $\mathcal{F}$ acts on the set $X$. If we associate the $5$-tuple $(t_1,t_2,\ldots,t_5)\in X$ to the hypergeometric function ${_{3}}F_{2}\left(\begin{array}{ccc}\chi_4^{t_1}, & \chi_4^{t_2}, & \chi_4^{t_3}\\ & \chi_4^{t_4}, & \chi_4^{t_5}\end{array}| 1\right)$, 
then each orbit of the group action consists of a number of $5$-tuples $(t_1,t_2,\ldots,t_5)$, and the corresponding ${}_3 F_{2}$ terms have the same value.
\end{lemma}
\begin{proof}
 For a proof, see Section $6$ of \cite{dawsey}.
\end{proof}
In order to prove Theorem \ref{asym}, the following famous theorem, due to Andr\'e Weil, serves as the crux. We state it here.
\begin{theorem}[Weil's estimate]\label{weil}
	Let $\mathbb{F}_q$ be the finite field of order $q$, and let $\chi$ be a character of $\mathbb{F}_q$ of order $s$. Let $f(x)$ be a polynomial of degree $d$ over $\mathbb{F}_q$ such that $f(x)$ cannot be written in the form $c\cdot {h(x)}^s$, where $c\in\mathbb{F}_q$. Then
	$$\Bigl\lvert\sum_{x\in\mathbb{F}_q}\chi(f(x))\Bigr\rvert\leq (d-1)\sqrt{q}.$$
\end{theorem}
The rest of the article goes as follows. In Section $3$, we prove Theorem \ref{thm1}. In Section $4$, we prove Theorem \ref{thm2}. Finally, in Section $5$ we prove the asymptotic formula for the number of cliques of any order in Peisert graphs.  
To count the number of cliques in Peisert graphs, we note that since the graph is vertex-transitive, so any two vertices in the graph are contained in the same number of cliques of a particular order.
We will also use the following notation throughout the proofs. For an induced subgraph $S$ of a Peisert graph and a vertex $v\in S$, we denote by $k_3(S)$ and $k_3(S,v)$ the number of cliques of order $3$ in $S$ and the number of cliques of order $3$ in $S$ containing $v$, respectively.
\section{number of $3$-order cliques in $P^\ast(q)$}
In this section, we prove Theorem \ref{thm1}. Recall that $\mathbb{F}_q^\ast=\langle g\rangle$ and $H=\langle g^4\rangle\cup g\langle g^4\rangle$. Also, $\langle H\rangle$ is the subgraph induced by $H$ and $h=1-\chi_4(g)$.
\begin{proof}[Proof of Theorem \ref{thm1}]
Using the vertex-transitivity of $P^\ast(q)$, we find that
\begin{align}\label{trian}
k_3(P^\ast(q))&=\frac{1}{3}\times q\times k_3(P^\ast(q),0)\notag \\
&=\frac{q}{3}\times \text{number of edges in }\langle H\rangle .
\end{align} 
Now, 
\begin{align}\label{ww-new}
\text{the number of edges in~} \langle H\rangle =\frac{1}{2}\times \mathop{\sum\sum}_{\chi_4(x-y)\in \{1, \chi_4(g)\}} 1,
\end{align}
where the 1st sum is taken over all $x$ such that $\chi_4(x)\in\{1,\chi_4(g)\}$ and the 2nd sum is taken over all $y\neq x$ such that $\chi_4(y)\in\{1,\chi_4(g)\}$. Hence, using \eqref{qq} in \eqref{ww-new}, we find that 
\begin{align}\label{ww}
&\text{the number of edges in~}\langle H\rangle \notag \\
&=\frac{1}{2\times 4^3}\sum\limits_{x\neq 0}(2+h\chi_4(x)+\overline{h}\overline{\chi_4}(x))\notag\\
&\hspace{1.5cm}\times \sum\limits_{y\neq 0,x}[(2+h\chi_4(y)+\overline{h}\overline{\chi_4}(y))(2+h\chi_4(x-y)+\overline{h}\overline{\chi_4}(x-y))].
\end{align}
We expand the inner summation in $\eqref{ww}$ to obtain
\begin{align}\label{ee}
&\sum\limits_{y\neq 0,x}[4+2h\chi_4(y)+2\overline{h}\overline{\chi_4}(y)+2h\chi_4(x-y)+2\overline{h}\overline{\chi_4}(x-y)+2\chi_4(y)\overline{\chi_4}(x-y)\notag \\
&	+2\overline{\chi_4}(y)\chi_4(x-y)-2\chi_4(g)\chi_4(y(x-y))+2\chi_4(g)\overline{\chi_4}(y(x-y))].
\end{align}
We have 
\begin{align}\label{new-eqn3}
\sum\limits_{y\neq 0,x}\chi_4(y(x-y))=\sum\limits_{y\neq 0,1}\chi_4(xy)\chi_4(x-xy)=\varphi(x) J(\chi_4,\chi_4).
\end{align}
Using Lemma \ref{lem2} and \eqref{new-eqn3}, \eqref{ee} yields
\begin{align}\label{new-eqn2}
&\sum\limits_{y\neq 0,x}[(2+h\chi_4(y)+\overline{h}\overline{\chi_4}(y))(2+h\chi_4(x-y)+\overline{h}\overline{\chi_4}(x-y))]\notag \\
&=4(q-3)-4h\chi_4(x)-4\overline{h}\overline{\chi_4}(x)-2\chi_4(g)\varphi(x)J(\chi_4,\chi_4)+2\chi_4(g)\varphi(x)\overline{J(\chi_4,\chi_4)}.
\end{align}
Now, putting \eqref{new-eqn2} into \eqref{ww}, and then using Lemma \ref{rr}, we find that 
\begin{align*}
&\text{the number of edges in }\langle H\rangle\\
=&\frac{1}{2\times 4^3}\sum\limits_{x\neq 0}[(2+h\chi_4(x)+\overline{h}\overline{\chi_4}(x))(4(q-3)-4h\chi_4(x)-4\overline{h}\overline{\chi_4}(x))]\\
=&\frac{1}{2\times 4^3}\sum\limits_{x\neq 0}[8(q-5)+(4h(q-3)-8h)\chi_4(x)+(4\overline{h}(q-3)-8\overline{h})\overline{\chi_4}(x)]\\
=&\frac{(q-1)(q-5)}{16}. 
\end{align*}
Substituting this value in $\eqref{trian}$ gives us the required result.
\end{proof}
\section{number of $4$-order cliques in $P^\ast(q)$}
In this section, we prove Theorem \ref{thm2}. First, we recall again that $\mathbb{F}_q^\ast=\langle g\rangle$ and $H=\langle g^4\rangle\cup g \langle g^4\rangle$. Let $J(\chi_4,\chi_4)=J(\chi_4,\varphi)=\rho$, where the value of $\rho$ is given by Lemma \ref{rr}. 
Let $q=u^2+2v^2$ for integers $u$ and $v$ such that $u\equiv 3\pmod 4$ and $p\nmid u$ when $p\equiv 3\pmod 8$. Let $\chi_8$ be a character of order $8$ such that $\chi_8^2=\chi_4$. Note that in the proof we shall use the fact that $\chi_4(-1)=1$ multiple times. Recall that $h=1-\chi_4(g)$.
\begin{proof}[Proof of Theorem \ref{thm2}]
Noting again that $P^\ast(q)$ is vertex-transitive, we find that
\begin{align}\label{tt}
k_4(P^\ast(q))
&=\frac{q}{4}\times \text{ number of $4$-order cliques in $P^\ast(q)$ containing }0\notag \\
&=\frac{q}{4}\times k_3(\langle H\rangle).
\end{align}
Let $a, b\in H$ be such that $\chi_4(ab^{-1})=1$. We note that
\begin{align}\label{new-eqn4}
k_3(\langle H\rangle, a) =\frac{1}{2}\times \mathop{\sum\sum}_{\chi_4(x-y)\in \{1, \chi_4(g)\}} 1,
\end{align}
where the 1st sum is taken over all $x$ such that $\chi_4(x), \chi_4(a-x)\in\{1,\chi_4(g)\}$ and the 2nd sum is taken over all $y\neq x$ such that $\chi_4(y), \chi_4(a-y)\in\{1,\chi_4(g)\}$. Hence, using \eqref{qq} in \eqref{new-eqn4}, we find that
\begin{align*}
&k_3(\langle H\rangle, a)\\
&=\frac{1}{2\times 4^5}\sum_{x\neq 0,a}\sum_{y\neq 0,a,x}[(2+h\chi_4(a-x)+\overline{h}\overline{\chi_4}(a-x))\\
&\times (2+h\chi_4(a-y)+\overline{h}\overline{\chi_4}(a-y))(2+h\chi_4(x-y)+\overline{h}\overline{\chi_4}(x-y))\\
&\times (2+h\chi_4(x)+\overline{h}\overline{\chi_4}(x))(2+h\chi_4(y)+\overline{h}\overline{\chi_4}(y))].
\end{align*}
Using the substitution $Y=ba^{-1}y$, the sum indexed by $y$ in the above yields
\begin{align*}
&k_3(\langle H\rangle, a)\\
&=\frac{1}{2\times 4^5}\sum_{x\neq 0,a}\sum_{Y\neq 0,b,ba^{-1}x}
[(2+h\chi_4(a-x)+\overline{h}\overline{\chi_4}(a-x))\\
&\times (2+h\chi_4(Y-b)+\overline{h}\overline{\chi_4}(Y-b))(2+h\chi_4(Y-ba^{-1}x)+\overline{h}\overline{\chi_4}(Y-ba^{-1}x))\\
&\times (2+h\chi_4(x)+\overline{h}\overline{\chi_4}(x))(2+h\chi_4(Y)+\overline{h}\overline{\chi_4}(Y))] \\
&=\frac{1}{2\times 4^5}\sum_{Y\neq 0,b}\sum_{x\neq 0,a,ab^{-1}Y}[(2+h\chi_4(a-x)+\overline{h}\overline{\chi_4}(a-x))\\
&\times
(2+h\chi_4(Y-b)+\overline{h}\overline{\chi_4}(Y-b)) (2+h\chi_4(Y-ba^{-1}x)+\overline{h}\overline{\chi_4}(Y-ba^{-1}x))\\
&\times (2+h\chi_4(x)+\overline{h}\overline{\chi_4}(x))(2+h\chi_4(Y)+\overline{h}\overline{\chi_4}(Y))].
\end{align*}
Again, using the substitution $X=ba^{-1}x$ yields 
\begin{align*}
&k_3(\langle H\rangle, a)\\
&=\frac{1}{2\times 4^5}\sum_{Y\neq 0,b}\sum_{X\neq 0,b,Y}[(2+h\chi_4(b-X)+\overline{h}\overline{\chi_4}(b-X))\\
&\times(2+h\chi_4(b-Y)+\overline{h}\overline{\chi_4}(b-Y))(2+h\chi_4(X-Y)+\overline{h}\overline{\chi_4}(X-Y))\\
&\times (2+h\chi_4(X)+\overline{h}\overline{\chi_4}(X))(2+h\chi_4(Y)+\overline{h}\overline{\chi_4}(Y))] \\
&=k_3(\langle H\rangle,b).     
\end{align*}
Thus, if $a, b\in H$ are such that $\chi_4(ab^{-1})=1$, then 
\begin{align}\label{cond}
k_3(\langle H\rangle,a)=k_3(\langle H\rangle,b).	
\end{align} 
Let $\langle g^4\rangle =\{x_1,\ldots,x_{\frac{q-1}{4}}\}$ with $x_1=1$ and $g\langle g^4\rangle=\{y_1,\ldots, y_{\frac{q-1}{4}}\}$ with $y_1=g$. Then,
\begin{align}\label{pick}
\sum_{i=1}^{\frac{q-1}{4}}k_3(\langle H\rangle,x_i)+\sum_{i=1}^{\frac{q-1}{4}}k_3(\langle H\rangle,y_i)=3\times k_3(\langle H\rangle).
\end{align}
By $\eqref{cond}$, we have  
$$k_3(\langle H\rangle,x_1)=k_3(\langle H\rangle,x_2)=\cdots=k_3(\langle H\rangle,x_{\frac{q-1}{4}})$$
and 
$$k_3(\langle H\rangle,y_1)=k_3(\langle H\rangle,y_2)=\cdots=k_3(\langle H\rangle,y_{\frac{q-1}{4}}).$$
Hence, \eqref{pick} yields 
\begin{align}\label{1g}
k_3(\langle H\rangle)=\frac{q-1}{12}[k_3(\langle H\rangle, 1)+ k_3(\langle H\rangle, g)].
\end{align}
Thus, we need to find only $k_3(\langle H\rangle, 1)$ and $k_3(\langle H\rangle, g)$. We first find $k_3(\langle H\rangle, 1)$. 
\par We have 
\begin{align}\label{xandy}
&k_3(\langle H\rangle,1)\notag \\
&=\frac{1}{2\times 4^5}\sum_{x\neq 0,1}[ (2+h\chi_4(1-x)+\overline{h}\overline{\chi_4}(1-x))(2+h\chi_4(x)+\overline{h}\overline{\chi_4}(x))]\notag\\
&\hspace{1.5cm} \sum_{y\neq 0,1,x}[(2+h\chi_4(1-y)+\overline{h}\overline{\chi_4}(1-y))
(2+h\chi_4(x-y)+\overline{h}\overline{\chi_4}(x-y)) \notag \\
&\hspace{2.5cm}\times  (2+h\chi_4(y)+\overline{h}\overline{\chi_4}(y))]. 
\end{align}
Let $i_1,i_2,i_3\in\{\pm 1\} $ and let $F_{i_1,i_2,i_3}$ denote the term $\chi_4^{i_1}(y)\chi_4^{i_2}(1-y)\chi_4^{i_3}(x-y)$. Using this notation, we expand and evaluate the inner summation in \eqref{xandy}. We have 
\begin{align}\label{sun}
&\sum_{y\neq 0,1,x}[2+h\chi_4(y)+\overline{h}\overline{\chi_4}(y)][2+h\chi_4(1-y)+\overline{h}\overline{\chi_4}(1-y)][2+h\chi_4(x-y)+\overline{h}\overline{\chi_4}(x-y)]\notag\\
&=\sum_{y\neq 0,1,x}[8+4h\chi_4(y)+4\overline{h}\overline{\chi_4}(y)+4h\chi_4(1-y)+4\overline{h}\overline{\chi_4}(1-y)+4h\chi_4(x-y)\notag\\&+4\overline{h}\overline{\chi_4}(x-y)
+4\chi_4(y)\overline{\chi_4}(1-y)+4\overline{\chi_4}(y)\chi_4(1-y)+4\chi_4(y)\overline{\chi_4}(x-y)\notag\\
&+4\overline{\chi_4}(y)\chi_4(x-y)
+4\chi_4(1-y)\overline{\chi_4}(x-y)+4\overline{\chi_4}(1-y)\chi_4(x-y)\notag\\
&+2h^2\chi_4(y)\chi_4(1-y)+2{\overline{h}}^2\overline{\chi_4}(y)\overline{\chi_4}(1-y)+2h^2\chi_4(y)\chi_4(x-y)\notag\\
&+2{\overline{h}}^2\overline{\chi_4}(y)\overline{\chi_4}(x-y)
+2h^2\chi_4(1-y)\chi_4(x-y)+2{\overline{h}}^2\overline{\chi_4}(1-y)\overline{\chi_4}(x-y)\notag\\
&+h^3 F_{1,1,1}+2hF_{1,1,-1}+2hF_{1,-1,1}+2\overline{h}F_{1,-1,-1}+2hF_{-1,1,1}+2\overline{h}F_{-1,1,-1}
\notag\\
&+2\overline{h}F_{-1,-1,1}+{\overline{h}}^3F_{-1,-1,-1}].
\end{align}
Now, referring to Lemmas \ref{lem1} and \ref{lem2}, we can easily check that any term of the form $\sum\limits_{y}\chi_4(\cdot)\overline{\chi_4}(\cdot)$ gives $-1$, $\sum\limits_y \chi_4((y-1)(y-x))$ 
gives $\varphi(x-1)\rho$ and $\sum\limits_y \chi_4(y(y-x))$ gives $\varphi(x)\rho$. Hence, $\eqref{sun}$ yields
\begin{align}\label{yonly}
&\sum_{y\neq 0,1,x}[2+h\chi_4(y)+\overline{h}\overline{\chi_4}(y)][2+h\chi_4(1-y)+\overline{h}\overline{\chi_4}(1-y)][2+h\chi_4(x-y)+\overline{h}\overline{\chi_4}(x-y)]\notag \\
&=A+B\chi_4(x)+\overline{B}\overline{\chi_4}(x)+B\chi_4(x-1)+\overline{B}\overline{\chi_4}(x-1)-4\chi_4(x)\overline{\chi_4}(x-1)\notag \\
&-4\overline{\chi_4}(x)\chi_4(x-1)-2h^2\chi_4(x)\chi_4(x-1)-2{\overline{h}}^2\overline{\chi_4}(x)\overline{\chi_4}(x-1)\notag \\
&+h^3 F_{1,1,1}+2hF_{1,1,-1}+2hF_{1,-1,1}+2\overline{h}F_{1,-1,-1}+2hF_{-1,1,1}+2\overline{h}F_{-1,1,-1}\notag
\\&+2\overline{h}F_{-1,-1,1}+{\overline{h}}^3F_{-1,-1,-1}\notag\\
&=:\mathcal{I},
\end{align}
where $A=8(q-8)$ and $B=-12h$. 
\par Next, we introduce some notations. Let 
\begin{align*}
B_1&=16(q-9)+6B+\overline{B}h^2,\\
D_1&=2\overline{B}-8\overline{h}+Bh^2-4h^3,\\
E_1&=8(q-9)+4Bh,\\
F_1&=16(q-9)+4 Re(B\overline{h}).
\end{align*} 
For $i\in\{1,2,3,4\}$ and $j\in\{1,2,\ldots,8\}$, we define the following character sums.
\begin{align*}
T_j&:=\sum_{x\neq 0,1}\sum_y \chi_4^{i_1}(y)\chi_4^{i_2}(1-y)\chi_4^{i_3}(x-y),\\
U_{ij}&:=\sum_{x\neq 0,1}\chi_4^l(m)\sum_y \chi_4^{i_1}(y)\chi_4^{i_2}(1-y)\chi_4^{i_3}(x-y),\\
V_{ij}&:=\sum_x\chi_4^{l_1}(x)\chi_4^{l_2}(1-x)\sum_y \chi_4^{i_1}(y)\chi_4^{i_2}(1-y)\chi_4^{i_3}(x-y),
\end{align*}
where 
\begin{align*}
l = \left\{
\begin{array}{lll}
1, & \hbox{if $i$ is odd,} \\
-1, & \hbox{\text{otherwise};}
\end{array}
\right.
\end{align*}
\begin{align*}
m = \left\{
\begin{array}{lll}
x, & \hbox{if $i\in\{1,2\}$,} \\
1-x, & \hbox{\text{otherwise;}}
\end{array}
\right.
\end{align*} 
and 
\begin{align*}
(l_1,l_2) = \left\{
\begin{array}{lll}
(1,1), & \hbox{if $i=1$,} \\
(1,-1), & \hbox{if $i=2$,} \\
(-1,1), & \hbox{if $i=3$,} \\
(-1,-1), & \hbox{if $i=4$.}
\end{array}
\right.
\end{align*} 
Also, corresponding to each $j$, let $(i_1,i_2,i_3)$ take the value according to the following table:
\begin{table}[h!]
	\begin{center}
		\begin{tabular}{ |c|  c|  c|  c| }
			\hline 
			$j$ & $i_1$ & $i_2$ & $i_3$ \\
			\hline
			$1$ & $1$ & $1$ & $1$ \\ 
			$2$ & $1$ & $1$ & $-1$ \\ 
			$3$ & $1$ & $-1$ & $1$\\ 
			$4$ & $1$ & $-1$ & $-1$\\ 
			$5$ & $-1$ & $1$ & $1$\\ 
			$6$ & $-1$ & $1$ & $-1$\\ 
			$7$ & $-1$ & $-1$ & $1$\\ 
			$8$ & $-1$ & $-1$ & $-1$\\ 
			\hline 
		\end{tabular}
	\end{center}
\end{table}\\
Then, using $\eqref{yonly}$ and the notations we just described, $\eqref{xandy}$ yields
\begin{align*}
&k_3(\langle H\rangle,1)=\frac{1}{2048}\sum_{x\neq 0,1}[2+h\chi_4(x)+\overline{h}\overline{\chi_4}(x)][2+h\chi_4(1-x)+\overline{h}\overline{\chi_4}(1-x)]\times \mathcal{I}\\
=&\frac{1}{2048}\sum_{x\neq 0,1}\Big[  32(q-15)+B_1\chi_4(x)+\overline{B_1}\overline{\chi_4}(x)+B_1\chi_4(x-1)+\overline{B_1}\overline{\chi_4}(x-1)\\
&+4 Re(Bh)\varphi(x)+4 Re(Bh)\varphi(x-1)+D_1\chi_4(x)\varphi(x-1)+\overline{D_1}\overline{\chi_4}(x)\varphi(x-1)\\
&+D_1\varphi(x)\chi_4(x-1)+\overline{D_1}\varphi(x)\overline{\chi_4}(x-1)+E_1\chi_4(x)\chi_4(x-1)+\overline{E_1}\overline{\chi_4}(x)\overline{\chi_4}(x-1) \\
&+F_1\chi_4(x)\overline{\chi_4}(1-x)+\overline{F_1}\overline{\chi_4}(x)\chi_4(x-1)\Big]\\
&+\frac{1}{2\times 4^5}\Big[ 4h^3T_1+8hT_2+8h T_3+8\overline{h}T_4+8h T_5+8\overline{h}T_6+8\overline{h}T_7+4{\overline{h}}^3 T_8\\ 
&+2h^4  U_{11}+4h^2 U_{12}+4h^2 U_{13}+8 U_{14}+4h^2 U_{15}+8 U_{16}+8 U_{17}+4{\overline{h}}^2 U_{18}\\
&+4h^2 U_{21} +8 U_{22}+8 U_{23}+4{\overline{h}}^2 U_{24}+8 U_{25}+4{\overline{h}}^2 U_{26}+4{\overline{h}}^2 U_{27}+2{\overline{h}}^4 U_{28}\\
&+2h^4  U_{31}+4h^2 U_{32}+4h^2 U_{33}+8 U_{34}+4h^2 U_{35}+8 U_{36}+8 U_{37}+4{\overline{h}}^2 U_{38}\\
&+4h^2 U_{41} +8 U_{42}+8 U_{43}+4{\overline{h}}^2 U_{44}+8 U_{45}+4{\overline{h}}^2 U_{46}+4{\overline{h}}^2 U_{47}+2{\overline{h}}^4 U_{48}\\
&+h^5 V_{11}+2h^3 V_{12}+2h^3 V_{13}+4h V_{14}+2h^3 V_{15}+4h V_{16}+4h V_{17}+4\overline{h} V_{18}\\
&+2h^3 V_{21}+4h V_{22}+4h V_{23}+4\overline{h}V_{24}+4h V_{25}+4\overline{h}V_{26}+4\overline{h}V_{27}+2{\overline{h}}^3 V_{28}\\
&+2h^3 V_{31}+4h V_{32}+4h V_{33}+4\overline{h}V_{34}+4h V_{35}+4\overline{h}V_{36}+4\overline{h}V_{37}+2{\overline{h}}^3 V_{38}\\
&+4h V_{41}+4\overline{h}V_{42}+4\overline{h}V_{43}+2{\overline{h}}^3 V_{44}+4\overline{h}V_{45}+2{\overline{h}}^3 V_{46}+2{\overline{h}}^3 V_{47}+{\overline{h}}^5 V_{48} \Big]. 	
\end{align*}
Using Lemmas \ref{lem3}, \ref{lema1} and \ref{corr}, we find that 
\begin{align}\label{bigex}
&k_3(\langle H\rangle,1)=\frac{1}{2048}\left[32(q^2-20q+81) \right.\notag \\
&+h^5 V_{11}+2h^3 V_{12}+2h^3 V_{13}+4h V_{14}+2h^3 V_{15}+4h V_{16}+4h V_{17}+4\overline{h} V_{18}\notag \\
&+2h^3 V_{21}+4h V_{22}+4h V_{23}+4\overline{h}V_{24}+4h V_{25}+4\overline{h}V_{26}+4\overline{h}V_{27}+2{\overline{h}}^3 V_{28}\notag \\
&+2h^3 V_{31}+4h V_{32}+4h V_{33}+4\overline{h}V_{34}+4h V_{35}+4\overline{h}V_{36}+4\overline{h}V_{37}+2{\overline{h}}^3 V_{38}\notag \\
&\left.+4h V_{41}+4\overline{h}V_{42}+4\overline{h}V_{43}+2{\overline{h}}^3 V_{44}+4\overline{h}V_{45}+2{\overline{h}}^3 V_{46}+2{\overline{h}}^3 V_{47}+{\overline{h}}^5 V_{48}\right]. 
\end{align}
Now, we convert each term of the form $V_{i j}$ $[i \in\{1,2,3,4\}, j\in\{1,2, \ldots, 8\}]$ into its equivalent $q^{2}\cdot {_{3}}F_{2}$ form. We use the notation $(t_{1}, t_{2}, \ldots, t_{5})\in \mathbb{Z}_4^5$ 
for the term $q^{2}\cdot {_{3}}F_{2}\left(\begin{array}{ccc}\chi_4^{t_{1}}, & \chi_4^{t_{2}}, & \chi_4^{t_{3}}\\ & \chi_4^{t_{4}}, & \chi_4^{t_{5}}\end{array}| 1\right)$.
Then, $\eqref{bigex}$ yields 
\begin{align}\label{bigexp}
&k_3(\langle H\rangle,1)=\frac{1}{2048}\left[32(q^2-20q+81)\notag \right. \\
&\hspace{.5cm}+h^{5}(3,1,1,2,2)+2 h^{3}(1,1,3,2,0)+2 h^{3}(3,1,1,0,2)+4 h(1,1,3,0,0)\notag \\
&\hspace{.5cm}+2h^{3}(3,3,1,0,2)+4 h(1,3,3,0,0)+4 h(3,3,1,2,2)+4 \overline{h}(1,3,3,2,0) \notag\\
&\hspace{.5cm}+2 h^{3}(3,1,3,2,2)+4 h(1,1,1,2,0)+4 h(3,1,3,0,2)+4 \overline{h}(1,1,1,0,0)\notag\\
&\hspace{.5cm}+4 h(3,3,3,0,2)+4 \overline{h}(1,3,1,0,0)+4 \overline{h}(3,3,3,2,2)+2 {\overline{h}}^{3}(1,3,1,2,0)\notag \\
&\hspace{.5cm}+2 h^{3}(3,1,3,2,0)+4 h(1,1,1,2,2)+4 h(3,1,3,0,0)+4 \overline{h}(1,1,1,0,2)\notag\\
&\hspace{.5cm}+4 h(3,3,3,0,0)+4 \overline{h}(1,3,1,0,2)+4 \overline{h}(3,3,3,2,0)+2 {\overline{h}}^{3}(1,3,1,2,2)  \notag \\
&\hspace{.5cm}+4 h(3,1,1,2,0)+4 \overline{h}(1,1,3,2,2)+4 \overline{h}(3,1,1,0,0)+2{\overline{h}}^{3}(1,1,3,0,2)\notag\\
&\hspace{.5cm}\left. +4 \overline{h}(3,3,1,0,0)+2 \overline{h}^{3}(1,3,3,0,2)+ 2\overline{h}^{3}(3,3,1,2,0)+{\overline{h}}^{5}(1,3,3,2,2)\right].
\end{align}
Next, we use Lemma \ref{dlemma} alongwith the notations therein. We list the tuples $(t_1,t_2,\ldots, t_5)$ in each orbit of the group action of $\mathcal{F}$ on $X$, and then group the corresponding terms in $\eqref{bigexp}$ together. 
The orbit representatives $(1,1,1,0,0)$, $(3,3,3,0,0)$, $(1,3,3,2,0)$, $(3,1,1,2,0)$ and $(1,1,3,0,0)$ given in the proof of Corollary 2.7 in \cite{dawsey} are the ones whose orbits exhaust the hypergeometric terms in $\eqref{bigexp}$. 
We denote the $q^2\cdot {_{3}}F_{2}$ terms corresponding to these orbit representatives as $M_1,M_2,\ldots,M_5$ respectively. Then, $\eqref{bigexp}$ yields
\begin{align}\label{mex}
&k_3(\langle H\rangle,1)=\frac{1}{2048}
\left[32(q^2-20q+81)\right. \notag \\
&\hspace{.5cm}+h^{5}M_4+2 h^{3}M_1+2 h^{3}M_1+4 hM_5	+2h^{3}M_1+4 hM_5+4 hM_1+4 \overline{h}M_3 \notag\\
&\hspace{.5cm}+2 h^{3}M_4+4 hM_5+4 hM_2+4 \overline{h}M_1+4 hM_5+4 \overline{h}M_5+4 \overline{h}M_5+2 {\overline{h}}^{3}M_3\notag \\
&\hspace{.5cm}+2 h^{3}M_4+4 hM_5+4 hM_5+4 \overline{h}M_5
+4 hM_2+4 \overline{h}M_1+4 \overline{h}M_5+2 {\overline{h}}^{3}M_3  \notag \\
&\hspace{.5cm}+4 hM_4+4 \overline{h}M_3+4 \overline{h}M_5+2{\overline{h}}^{3}M_2
+4 \overline{h}M_5+2 \overline{h}^{3}M_2+\left. 2\overline{h}^{3}M_2+{\overline{h}}^{5}M_3\right].
\end{align}
Using Lemma \ref{lem4} (note that we could not reduce $M_5$), $\eqref{mex}$ yields
\begin{align}\label{mexp}
k_3(\langle H\rangle,1)=\frac{1}{128}\left[2(q^2-20q+81)+2 u(-p)^t
+3q^2\cdot{_{3}}F_{2}\left(\begin{array}{ccc}\chi_4, & \chi_4, & \overline{\chi_4}\\ & \varepsilon, & \varepsilon\end{array}| 1\right) \right].	
\end{align}
Returning back to $\eqref{1g}$, we are now left to calculate $k_3( \langle H\rangle,g)$. Again, we have 
\begin{align}\label{gxandy}
&k_3(\langle H\rangle,g)\notag\\
&=\frac{1}{2048}\sum_{x\neq 0,g}\sum_{y\neq 0,g,x}\left[ (2+h\chi_4(g-x)+\overline{h}\overline{\chi_4}(g-x)) (2+h\chi_4(g-y)+\overline{h}\overline{\chi_4}(g-y))\notag \right.\\
&\times\left. (2+h\chi_4(x-y)+\overline{h}\overline{\chi_4}(x-y))(2+h\chi_4(x)+\overline{h}\overline{\chi_4}(x)) (2+h\chi_4(y)+\overline{h}\overline{\chi_4}(y))\right]. 
\end{align}	
Using the substitutions $Y=yg^{-1}$ and $X=xg^{-1}$, and then using the fact that $h\chi_4(g)=\overline{h}$, \eqref{gxandy} yields
\begin{align*}
&k_3(\langle H\rangle,g)\\
&=\frac{1}{2048}\sum_{x\neq 0,1}\sum_{y\neq 0,1,x}\left[ (2+\overline{h}\chi_4(1-x)+h\overline{\chi_4}(1-x)) (2+\overline{h}\chi_4(1-y)+h\overline{\chi_4}(1-y))\notag \right.\\
&\times\left. (2+\overline{h}\chi_4(x-y)+h\overline{\chi_4}(x-y)) (2+\overline{h}\chi_4(x)+h\overline{\chi_4}(x))(2+\overline{h}\chi_4(y)+h\overline{\chi_4}(y))\right].
\end{align*}
Comparing this with $\eqref{xandy}$ we see that the expansion of the expression inside this summation will consist of the same summation terms as in $\eqref{xandy}$ except that the coefficient corresponding to each summation will 
become the complex conjugate of the corresponding coefficient of the same summation. This means that, to calculate the coefficient of each summation after expanding the expression in $\eqref{gxandy}$, we need to replace each corresponding coefficient in $\eqref{mex}$ by its complex conjugate. 
Now, $\eqref{mexp}$ is the final expression from $\eqref{mex}$, and we see that \eqref{mexp} contains three summands, two of them being real numbers and the other being a ${}_3F_2$ term whose coefficient is also a real number. Then by the foregoing argument, $\eqref{gxandy}$ yields the same value as given in $\eqref{mexp}$. Thus, $\eqref{1g}$ gives that
\begin{align*}
k_3(\langle H\rangle)=\frac{q-1}{768}&\left[2(q^2-20q+81)+2 u(-p)^t+3q^2\cdot  {_{3}}F_{2}\left(\begin{array}{ccc}\chi_4, & \chi_4, & \overline{\chi_4}\\ & \varepsilon, & \varepsilon\end{array}| 1\right)\right].	
\end{align*}
Substituting the above value in $\eqref{tt}$, we complete the proof of the theorem.
\end{proof}
\section{proof of theorem $\ref{asym}$}
Let $m\geq 1$ be an integer. We have observed that the calculations for computing the number of $4$-order cliques in $P^\ast(q)$ become very tedious. 
However, we can have an asymptotic result on the number of cliques of order $m$ in $P^\ast(q)$ as $q\rightarrow\infty$. The method follows along the lines of \cite{wage} and so we prove by the method of induction.
\begin{proof}[Proof of Theorem \ref{asym}]
	Let $\mathbb{F}_q^\ast=\langle g\rangle$. We set a formal ordering of the elements of $\mathbb{F}_q:\{a_1<\cdots<a_q\}$. Let $\chi_4$ be a fixed character on $\mathbb{F}_q$ of order $4$ and let $h=1-\chi_4(g)$. 
	First, we note that the result holds for $m=1,2$ and so let $m\geq 3$. Let the induction hypothesis hold for $m-1$. We shall use the notation `$a_m\neq a_i$' to mean $a_m\neq a_1,\ldots,a_{m-1}$. Recalling \eqref{qq}, we see that
	\begin{align}\label{ss}
	k_m(P^\ast(q))&=\mathop{\sum\cdots\sum}_{a_1<\cdots<a_m}\prod_{1\leq i<j\leq m} \frac{2+h\chi_4(a_i-a_j)+\overline{h}\chi_4^3(a_i-a_j)}{4}\notag \\
	&=\frac{1}{m}\mathop{\sum\cdots\sum}_{a_1<\cdots<a_{m-1}}\left[  \prod_{1\leq i<j\leq m-1}\frac{2+h\chi_4(a_i-a_j)+\overline{h}\chi_4^3(a_i-a_j)}{4}\right.\notag \\
	&\left.\frac{1}{4^{m-1}}\sum\limits_{a_m\neq a_i}\prod_{i=1}^{m-1}\{2+h\chi_4(a_m-a_i)+\overline{h}\chi_4^3(a_m-a_i)\}\right] 
	\end{align}
	In order to use the induction hypothesis, we try to bound the expression $$\sum\limits_{a_m\neq a_i}\prod_{i=1}^{m-1}\{2+h\chi_4(a_m-a_i)+\overline{h}\chi_4^3(a_m-a_i)\}$$
	in terms of $q$ and $m$. We find that 
	\begin{align}\label{dd}
	\mathcal{J}&:=\sum\limits_{a_m\neq a_i} \prod_{i=1}^{m-1}\{2+h\chi_4(a_m-a_i)+\overline{h}\chi_4^3(a_m-a_i)\}\notag \\
	&=2^{m-1}(q-m+1)\notag \\
	&+\sum\limits_{a_m\neq a_i}[(3^{m-1}-1)\text{ number of terms containing expressions in }\chi_4]
	\end{align}
	Each term in \eqref{dd} containing $\chi_4$ is of the form $$2^f h^{i'}\overline{h}^{j'}\chi_4((a_m-a_{i_1})^{j_1}\cdots (a_m-a_{i_s})^{j_s}),$$ where 
	\begin{equation}\label{asy}
		\left.\begin{array}{l}
			0\leq f\leq m-2,\\
			0\leq i',j'\leq m-1,\\
			i_1,\ldots,i_s \in \{1,2,\ldots,m-1\},\\
			j_1,\ldots,j_s \in \{1,3\},\text{ and}\\
			1\leq s\leq m-1.
		\end{array}\right\}
	\end{equation}
Let us consider such an instance of a term containing $\chi_4$. Excluding the constant factor $2^fh^{i'}\overline{h}^{j'}$, we obtain a polynomial in the variable $a_m$. Let $g(a_m)=(a_m-a_{i_1})^{j_1}\cdots (a_m-a_{i_s})^{j_s}$. Using Weil's estimate (Theorem \ref{weil}), we find that
\begin{align}\label{asy1}
\mid\sum\limits_{a_m\in\mathbb{F}_q}\chi_4(g(a_m))\mid\leq (j_1+\cdots+j_s-1)\sqrt{q}.	
\end{align}
Then, using \eqref{asy1} we have
\begin{align}\label{asy2}
	|2^fh^{i'}\overline{h}^{j'} \sum\limits_{a_m}\chi_4(g(a_m))|&\leq 2^{f+i'+j'}(j_1+\cdots+j_s-1)\sqrt{q}\notag \\
	&\leq 2^{3m-4}(3m-4)\sqrt{q}\notag \\
	&\leq 2^{3m}\cdot 3m\sqrt{q}.
\end{align}
Noting that the values of $\chi_4$ are roots of unity, using \eqref{asy2}, and using \eqref{asy} and the conditions therein, we obtain
	\begin{align*}
	&\mid 2^f h^{i'}\overline{h}^{j'}\sum\limits_{a_m\neq a_i}\chi_4(g(a_m))\mid\\
	&=\mid 2^fh^{i'}\overline{h}^{j'}\left\lbrace \sum\limits_{a_m}\chi_4(g(a_m))-\chi_4(g(a_1))-\cdots-\chi_4(g(a_{m-1})) \right\rbrace \mid\\
	&\leq 2^{3m}\cdot 3m\sqrt{q}+2^{2m-3}\\ 
	&\leq 2^{2m}(1+2^m\cdot 3m\sqrt{q}),
	\end{align*}
	that is,
	$$-2^{2m}(1+2^m\cdot 3m\sqrt{q})\leq 2^f h^{i'}\overline{h}^{j'}\sum\limits_{a_m\neq a_i}\chi_4(g(a_m))\leq 2^{2m}(1+2^m\cdot 3m\sqrt{q}).$$
	Then, \eqref{dd} yields
	\begin{align*}
	&2^{m-1}(q-m+1)-2^{2m}(1+2^m\cdot 3m\sqrt{q})(3^{m-1}-1)\\
	&\leq \mathcal{J}\\
	&\leq 2^{m-1}(q-m+1)+2^{2m}(1+2^m\cdot 3m\sqrt{q})(3^{m-1}-1)	
	\end{align*}
	and thus, \eqref{ss} yields
	\begin{align}\label{asy3}
	&[2^{m-1}(q-m+1)-2^{2m}(1+2^m\cdot 3m\sqrt{q})(3^{m-1}-1)]\times\frac{1}{m\times 4^{m-1}}k_{m-1}(P^\ast(q))\notag\\
	&\leq k_m(P^\ast(q))\notag \\
	&\leq [2^{m-1}(q-m+1)+2^{2m}(1+2^m\cdot 3m\sqrt{q})(3^{m-1}-1)]\times\frac{1}{m\times 4^{m-1}}k_{m-1}(P^\ast(q))
\end{align}
	Dividing by $q^m$ throughout in \eqref{asy3} and taking $q\rightarrow \infty$, we have
	\begin{align}\label{ff}
	&\lim_{q\rightarrow \infty}\frac{2^{m-1}(q-m+1)-2^{2m}(1+2^m\cdot 3m\sqrt{q})(3^{m-1}-1)}{m\times 4^{m-1}\times q}\lim_{q\rightarrow \infty}\frac{k_{m-1}(P^\ast(q))}{q^{m-1}}\notag \\ 
	&\leq \lim_{q\rightarrow \infty}\frac{k_m(P^\ast(q))}{q^m}\notag \\
	&\leq \lim_{q\rightarrow \infty}\frac{2^{m-1}(q-m+1)+2^{2m}(1+2^m\cdot 3m\sqrt{q})(3^{m-1}-1)}{m\times 4^{m-1}\times q}\lim_{q\rightarrow \infty}\frac{k_{m-1}(P^\ast(q))}{q^{m-1}}
	\end{align}
	Now, using the induction hypothesis and noting that
	\begin{align*}
	&\lim\limits_{q\to\infty}\frac{2^{m-1}(q-m+1)\pm 2^{2m}(1+2^m\cdot 3m\sqrt{q})(3^{m-1}-1)}{m\times 4^{m-1}q}\\
	&=\frac{1}{m\times 4^{m-1}}2^{m-1}\\
	&=\frac{1}{m\times 2^{m-1}}	,
	\end{align*} 
	we find that both the limits on the left hand side and the right hand side of \eqref{ff} are equal. This completes the proof of the result.	 
\end{proof}
Taking $m=3$ in Theorem \ref{asym}, we find that 
$$\lim\limits_{q\to\infty}\dfrac{k_3(P^\ast(q))}{q^3}=\frac{1}{48}.$$
We obtain the same limiting value from Theorem \ref{thm1} as well.
\par Taking $m=4$ in Theorem \ref{thm2} and Theorem \ref{asym}, we obtain the following corollary which is also evident from Table \ref{Table-1}.
\begin{corollary}\label{cor1}
	We have 
	\begin{align*}
	\lim\limits_{q\to\infty} {_{3}}F_{2}\left(\begin{array}{ccc}
	\chi_4, & \chi_4, & \chi_4^3 \\
	& \varepsilon, & \varepsilon
	\end{array}| 1\right)=0.
	\end{align*}
\end{corollary}
\begin{proof}
Putting $m=4$ in Theorem \ref{asym}, we have 
\begin{align}\label{eqn1-cor1}
\lim\limits_{q\to\infty}\dfrac{k_4(P^\ast(q))}{q^4}=\frac{1}{1536}.
\end{align}
Putting $m=4$ in Theorem \ref{thm2}, we have 
\begin{align}\label{eqn2-cor1}
\lim\limits_{q\to\infty}\dfrac{k_4(P^\ast(q))}{q^4}=\frac{1}{1536}+3\times \lim\limits_{q\to\infty} {_{3}}F_{2}\left(\begin{array}{ccc}
\chi_4, & \chi_4, & \chi_4^3 \\
& \varepsilon, & \varepsilon
\end{array}| 1\right).
\end{align}
Combining \eqref{eqn1-cor1} and \eqref{eqn2-cor1}, we complete the proof.
\end{proof}
\section{Acknowledgements}
We are extremely grateful to Ken Ono for previewing a preliminary version of this paper and for his helpful comments.

\end{document}